\documentclass[11pt,reqno]{amsart}
\usepackage[T1]{fontenc}
\usepackage{graphicx}
\usepackage{color}
\definecolor{MyLinkColor}{rgb}{0,0,0.4}
\newcommand{\R}{{\mathbb R}}
\newcommand{\E}{{\mathcal E}}

\newcommand{\N}{{\mathbb N}}

\newcommand{\B}{\mathcal{B}}

\newcommand{\V}{\mathcal{V}}

\newcommand{\kL}{\mathcal{L}}

\newcommand{\wh}{\widehat}

\newcommand{\p}{\partial}
\newcommand{\A}{\mathcal{A}}
\newcommand{\e}{\varepsilon}

\newtheorem{thm}{Theorem}[section]

\newtheorem{lemma}[thm]{Lemma}

\theoremstyle{remark} 
\newtheorem{rem}[thm]{Remark}
\numberwithin{equation}{section}

\title[A degenerate parabolic system modelling  flows in porous media]{Existence and stability of weak solutions for a degenerate parabolic system modelling two-phase flows in porous media}
\thanks{ Partially supported by the French-German procope project 20190SE}

\author[J. Escher]{Joachim Escher}
\address{Institut f{\"u}r Angewandte Mathematik, Leibniz Universit{\"a}t Hannover, Welfengarten~1, 30167 Hannover, Germany.}
\email{escher@ifam.uni-hannover.de}

\author[Ph. Lauren\c cot]{Philippe Lauren\c cot}
\address{Institut  de Math\'ematiques de Toulouse, CNRS UMR 5219, Universit\'e de Toulouse, F-31062  Toulouse cedex 9, France}
\email{laurenco@math.univ-toulouse.fr}

\author[B.--V. Matioc]{Bogdan--Vasile Matioc}
\address{Institut f{\"u}r Angewandte Mathematik, Leibniz Universit{\"a}t Hannover, Welfengarten~1, 30167 Hannover, Germany.}
\email{matioc@ifam.uni-hannover.de}

\subjclass[2010]{ 35K65; 35K40; 35D30; 35B35;  35Q35}
\keywords{Degenerate parabolic system,  weak solutions, exponential stability, thin film, { Liapunov functional}}

\usepackage[colorlinks=true,linkcolor=MyLinkColor,citecolor=MyLinkColor]{hyperref} %dvips

\begin{document}

\begin{abstract}
We prove global existence of nonnegative weak solutions to a degenerate parabolic system which models the interaction of two thin fluid films in a porous medium. Furthermore, we show that these weak solutions converge at an exponential rate towards flat equilibria.
\end{abstract}

\maketitle

%%% SECTION: INTRO %%%

%%%%%%%%%%%%%%%%%%%%%%%%%%%%%%%%%%%%%%%%
%%%%%%%%%%%%%%%%%%%%%%%%%%%%%%%%%%%%%%%%
\section{Introduction}
%%%%%%%%%%%%%%%%%%%%%%%%%%%%%%%%%%%%%%%%
%%%%%%%%%%%%%%%%%%%%%%%%%%%%%%%%%%%%%%%%

In this paper we consider the following system of degenerate parabolic equations 
\begin{equation}\label{eq:S1}
\left\{
\begin{array}{llll}
\p_t f=&\p_x\left(   f\p_xf\right)+R\p_x\left(f\p_xh\right),\\[1ex]
\p_th=&\p_x\left(   f\p_xf\right)+R_\mu\p_x\left[  (h-f) \p_xh\right]+R \p_x\left(  f\p_xh\right),
\end{array}
\right.
{ \quad (t,x)\in (0,\infty)\times (0,L),}
\end{equation}
which models two-phase flows in porous media { under the assumption that the thickness of the two fluid layers is small.
Indeed, the} system \eqref{eq:S1} has been obtained in \cite{EMM2}  by passing to the limit of small layer thickness  in the Muskat problem studied in \cite{EMM1} (with  homogeneous Neumann boundary condition).
Similar methods to those presented in \cite{EMM2} have been used in \cite{GP} and \cite{MP}, 
where it is rigorously shown that, in the absence of gravity, appropriate scaled classical solutions of the Stokes' and  one-phase Hele-Shaw  problems with surface tension  converge to solutions of  thin film equations
\begin{equation*}
\partial_th+\p_x(h^a\p_x^3h)=0,
\end{equation*} 
with $a=3$ for Stoke's problem and $a=1$ for the Hele-Shaw problem.

In our setting $f$ is a nonnegative function expressing the height of the interface between the fluids { while $h\geq f$ is the height of the interface separating the fluid located in the upper part of the porous medium from air,} cf. Figure~\ref{F:1}. 
\begin{figure}
$$\includegraphics[width=10cm]{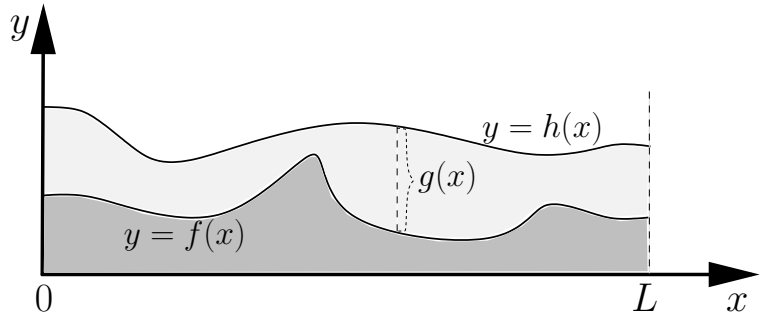}$$
%$$\includegraphics[width=10cm]{Figure.pdf}$$
\caption{The physical setting}
\label{F:1}
\end{figure}
We assume that the bottom of the porous medium, which is located at  $y=0,$  is impermeable and that the air is at constant  pressure normalised to be zero.  { The parameters $R$ and $R_\mu$ are given by 
\[
R:=\frac{\rho_+}{\rho_--\rho_+}\quad\text{and}\quad R_\mu:=\frac{\mu_-}{\mu_+}R\,,
\]
where $\rho_-$, $\mu_-$ [resp. $\rho_+$, $\mu_+$] denote the density and viscosity of the fluid located below [resp. above] in the porous medium.} Of course, we have to supplement system \eqref{eq:S1} with initial conditions
\begin{equation}\label{eq:bc1}
 f(0)=f_0,\qquad h(0)=h_0, { \quad x\in (0,L),}
 \end{equation}
and we impose no-flux boundary conditions at $x=0$ and $x=L$:
\begin{equation}\label{eq:bc2}
 \p_xf=\p_xh=0,\qquad x=0, L.
 \end{equation}
It turns out that  the system is parabolic if we assume that $h_0>f_0>0$ and $R>0,$ that is, the  denser fluid lies beneath. 
Existence and uniqueness of classical solutions to \eqref{eq:S1} have been established in this parabolic setting  in \cite{EMM2}.
Furthermore, it is also shown that the steady states of \eqref{eq:S1} are flat and that they attract at 
 an exponential rate in $H^2$ solutions which are initially close by.

In this paper we are interested in the degenerate case which appears when we allow $f_0=0$  and $h_0=f_0$ on some subset of $(0,L).$
 Owing to the loss of uniform parabolicity, existence of classical solutions can no longer be established by using  parabolic theory and we have to work within an appropriate  weak setting.
Furthermore, the system  is quasilinear and, as a further difficulty, each equation contains highest order derivatives of both unknowns { $f$ and $h$,} i.e. it is strongly coupled. 
In order to study the problem \eqref{eq:S1} we shall employ some  of the methods used { in \cite{Su2}} to investigate the spreading of insoluble surfactant.
However,  in our case
the situation is more involved since we have two sources of degeneracy, namely when  $f$ and $g:=h-f$ become zero.
It turns out that by choosing  $(f,g)$ as unknowns, the system \eqref{eq:S1} is more symmetric: 
 \begin{equation}\label{eq:S2}
\left\{
\begin{array}{llll}
\p_t f=&(1+R)\p_x\left(   f\p_xf\right)+R\p_x\left(f\p_xg\right),\\[1ex]
\p_tg=&R_\mu\p_x\left(   g\p_xf\right)+R_\mu \p_x\left(g\p_xg\right),
\end{array}
\right.
{  \quad (t,x)\in (0,\infty)\times (0,L),}
\end{equation}
since, up to  multiplicative constants, the first equation can be obtained from the second by simply interchanging $f$ and $g$.
Corresponding to  \eqref{eq:S2} we introduce the following energy functionals:
\[
 \E_1(f,g):=\int_0^L \left[ (f\ln f-f+1)+\frac{R}{R_\mu}(g\ln g-g+1) \right]\, dx
\]
and 
\[
\E_2(f,g):=\int_0^L \left[ f^2+R(f+g)^2 \right]\, dx.
\]
It is not difficult to see that both energy functionals $\E_1$ and $\E_2$ dissipate along classical solutions of \eqref{eq:S2}.
While in the classical setting the functional $\E_2$ plays an important role in the study of the stability properties of equilibria \cite{EMM2}, in the weak setting we strongly rely on the weaker energy $\E_1$  which,  nevertheless, provides us with suitable estimates for solutions of a regularised problem and enables us to pass to the limit to  obtain weak solutions. Note also that $\E_1$ appears quite natural in the context of \eqref{eq:S2}, while, when considering \eqref{eq:S1}, one would not expect to have an energy functional of this form.

Our main results read as follows:

%%%%%%%%%%%%%%%%%%%%%%%%%%%%%%%%%%%%%%%%
\begin{thm}\label{T:1} Assume that $R>0,$ $R_\mu>0.$ 
Given $f_0,g_0\in L_2{ ((0,L))}$ with $f_0\geq 0$ and $ g_0\geq 0$  there exists a global { weak} solution $(f,g) $ of \eqref{eq:S2} satisfying
 \begin{itemize}
\item[$(i)$] {$f\ge 0$, $g\ge 0$ in $(0,T)\times (0,L),$}
\item[$(ii)$] {$f,g\in L_\infty((0,T),L_2((0,L)))\cap L_2((0,T), H^1((0,L))),$}
\end{itemize}
for all $T>0$ and 
\begin{align*}
&\int_0^L f(T)\psi\, dx-\int_0^L f_0\psi\, dx= { -} \int_0^T\int_0^L \left((1+R)f\p_xf+R f\p_xg\right)\p_x\psi\, dx\, dt,\\[1ex]
&\int_0^L g(T)\psi\, dx-\int_0^L g_0\psi\, dx={ -} R_\mu\int_0^T\int_0^L \left(g\p_xf+ g\p_xg\right)\p_x\psi\, dx\, dt
\end{align*}
for all $\psi\in W^1_\infty((0,L)).$
Moreover, the weak solutions satisfy 
\begin{align*}
(a)\quad & { \|f(T)\|_1=\|f_0\|_1, \|g(T)\|_1=\|g_0\|_1,} \\[1ex]
(b) \quad &\E_1(f(T),g(T))+\int_0^T\int_0^L \left[ \frac{1}{2} |\p_xf|^2 + \frac{R}{1+2R} |\p_xg|^2 \right]\, dx\, dt\leq { \E_1(f_0,g_0),}\\[1ex]
(c)\quad &\E_2(f(T), g(T))+\int_0^T\int_0^L \left[ f\left((1+R)\p_xf+R\p_xg\right)^2+RR_\mu g(\p_xf+\p_xg)^2 \right]\, dx\, dt\leq \E_2(f_{0},g_{0})
\end{align*}
for almost all  $T\in(0,\infty)$.
\end{thm}
%%%%%%%%%%%%%%%%%%%%%%%%%%%%%%%%%%%%%%%%

%%%%%%%%%%%%%%%%%%%%%%%%%%%%%%%%%%%%%%%%
{ \begin{rem}\label{rem:1}
If $f_0=0$ for instance, a solution to \eqref{eq:S2} is $(0,g)$ where $g$ solves the classical porous medium equation $\partial_t g = R_\mu \partial_x\left( g \partial_x g \right)$ in $(0,\infty)\times (0,L)$ with homogeneous Neumann boundary conditions and initial condition $g_0$.
\end{rem} }
%%%%%%%%%%%%%%%%%%%%%%%%%%%%%%%%%%%%%%%%

Additionally to the existence result, we show that the weak solutions constructed in Theorem~\ref{T:1} converge at an exponential rate
towards the unique flat equilibrium (which is determined by mass conservation) in the $L_2-$norm:   
 
%%%%%%%%%%%%%%%%%%%%%%%%%%%%%%%%%%%%%%%%
 \begin{thm}[Exponential stability]\label{T:2} Under the assumptions of Theorem \ref{T:1}, there exist positive constants $M$ and $\omega$ such that
 \[
\left\|f(t)-\frac{1}{L}\int_0^Lf_0\, dx\right\|_{2}^2+\left\|g(t)-\frac{1}{L}\int_0^Lg_0\, dx\right\|_{2}^2\leq M e^{-\omega t}\qquad \text{for a.e. $t\geq0.$}
\]
\end{thm}
%%%%%%%%%%%%%%%%%%%%%%%%%%%%%%%%%%%%%%%%

%%%%%%%%%%%%%%%%%%%%%%%%%%%%%%%%%%%%%%%%
\begin{rem}\label{rem:2} Theorem~\ref{T:2} suggests that degenerate solutions become classical after  evolving over a certain  finite period of time, and therefore would converge in the $H^2-$norm towards the corresponding equilibrium, cf. \cite{EMM2}.   
\end{rem}
%%%%%%%%%%%%%%%%%%%%%%%%%%%%%%%%%%%%%%%%

The outline of the paper is as follows. In Section~\ref{sec:1} we regularise the system \eqref{eq:S2} and prove that { the regularised system has global classical solutions,} the global existence being a consequence of their boundedness away from zero and in $H^1((0,L)).$ { The purpose of the regularisation is twofold: on the one hand, the regularised system is expected to be uniformly parabolic and this is achieved by modifying \eqref{eq:S2} and the initial data such that the comparison principle applied to each equation separately guarantees that $f\ge\e$ and $g\ge \e$ for some $\e>0$. On the other hand, the regularised system is expected to be weakly coupled, a property which is satisfied by a suitable mollification of $\partial_x g$ in the first equation of \eqref{eq:S2} and $\partial_x f$ in the second equation of \eqref{eq:S2}. The energy functional $\mathcal{E}_1$ turns out to provide useful estimates for the regularised system as well.} In Section~\ref{sec:2} we show that the  classical global solutions of the regularised problem converge, in  appropriate norms, towards a weak  solution of \eqref{eq:S2}, and  that they satisfy similar energy inequalities as the classical solutions of \eqref{eq:S2} for both energy functionals $\E_1$ and $\E_2.$ Finally, we give a detailed proof of Theorem~\ref{T:2}. 

\medskip

Throughout the paper, we set $L_p:=L_p((0,L))$ and $W^1_p:=W^1_p((0,L))$ for $p\in [1,\infty]$, and $H^{2\alpha}:=H^{2\alpha}((0,L))$ for $\alpha\in [0,1]$. We also denote positive constants that may vary from line to line and depend only on $L$, $R$, $R_\mu$, and $(f_0,g_0)$ by $C$ or $C_i$, $i\ge 1$. The dependence of such constants upon additional parameters will be indicated explicitly.
 
%%%%%%%%%%%%%%%%%%%%%%%%%%%%%%%%%%%%%%%%
%%%%%%%%%%%%%%%%%%%%%%%%%%%%%%%%%%%%%%%%
\section{The regularised system}\label{sec:1}
%%%%%%%%%%%%%%%%%%%%%%%%%%%%%%%%%%%%%%%%
%%%%%%%%%%%%%%%%%%%%%%%%%%%%%%%%%%%%%%%%

In this section we introduce a regularised system which possesses global solutions provided they are bounded  in $H^1$ and also bounded away from zero. In Section~\ref{sec:2} we show that these solutions converge towards  weak solutions of \eqref{eq:S2}.

\medskip
We fix two nonnegative functions $f_0$ and $g_0$ in $L_2$ (the initial data of  system \eqref{eq:S2}) and first introduce the space
\[
H^2_\B:=\{f\in H^2((0,L))\,:\, \p_x f(0)=\p_xf(L)=0\}.
\] 
We note that for each $\e>0,$ the elliptic operator  $(1-\e^2\p_x^2):H^2_\B\to L_2$ is an isomorphism.
This property is preserved when considering the  restriction 
\[
(1-\e^2\p_x^2):\{f\in C^{2+\alpha}([0,L])\,:\,\p_x f(0)=\p_xf(L)=0\} \to C^{\alpha}([0,L]), \qquad \alpha\in (0,1).
\]
Given $f,g\in L_2,$ we let then 
\begin{equation}
F_\e:=(1-\e^2\p_x^2)^{-1}f, \qquad G_\e:=(1-\e^2\p_x^2)^{-1}g,
\label{spirou}
\end{equation}
and consider the following regularised problem
\begin{subequations}\label{eq:RS}
 \begin{equation}\label{eq:A1}
\left\{
\begin{array}{llll}
\p_t f_\e=&(1+R)\p_x\left(   f_\e\p_x{f_\e}\right)+R\p_x\left((f_\e-\e)\p_x{G_\e}\right),\\[1ex]
\p_tg_\e=&R_\mu\p_x\left(   (g_\e-\e)\p_x{F_\e}\right)+R_\mu \p_x\left(g_\e\p_x{g_\e}\right),
\end{array}
\right.
{ \quad (t,x)\in (0,\infty)\times (0,L),}
\end{equation}
supplemented with homogeneous Neumann boundary conditions  
\begin{equation}\label{A:2}
\p_xf_\e=\p_x g_\e=0\qquad x=0,L,
\end{equation}
and with regularised initial data
\begin{equation}\label{eq:A3}
f_\e(0)=f_{0\e}:=(1-\e^2\p_x^2)^{-1}f_0+\e,\qquad g_\e(0)=g_{0\e}:=(1-\e^2\p_x^2)^{-1}g_0+\e.
\end{equation}
\end{subequations}
 Note that the regularised initial data $(f_{0\e}, g_{0\e})\in H^2_\B\times H^2_\B$ and, invoking the elliptic maximum principle, we have
\begin{equation}\label{eq:id0} 
f_{0\e}\geq \e,\qquad  g_{0\e}\geq\e.
\end{equation}
 Letting $F_{0\e}:=(1-\e^2\p_x^2)^{-1}f_0$ and $G_{0\e}:=(1-\e^2\p_x^2)^{-1}g_0,$
 we obtain by multiplying the relation $F_{0\e}-\e^2\p_x^2 F_{0\e}=f_0$ by $F_{0\e}$ and integrating over $(0,L)$ the following relation
 \[
\|F_{0\e}\|_2^2+\e^2\|\p_x F_{0\e}\|_2^2=\int_0^Lf_0F_{0\e}\, dx { \le \|f_0\|_2\ \|F_{0\e}\|_2,}
\]
which gives a uniform $L_2$-bound for the regularised initial data:
  \begin{equation}\label{eq:id1}
 \|f_{0\e}\|_2\leq \|f_0\|_2+\e { \sqrt{L}}\qquad\text{and}\qquad \|g_{0\e}\|_2\leq \|g_0\|_2+\e { \sqrt{L}}.
  \end{equation}
 
\medskip

Concerning the solvability of problem \eqref{eq:RS}, we use quasilinear parabolic theory, as presented in \cite{Am2}, to prove the following result:

%%%%%%%%%%%%%%%%%%%%%%%%%%%%%%%%%%%%%%%%
\begin{thm}\label{T:3} For each $\e\in(0,1)$ problem~\eqref{eq:RS} possesses a unique global nonnegative solution $X_\e:=(f_\e,g_\e)$ with 
\[
f_\e, g_\e\in C([0,\infty), H^1)\cap C((0,\infty), H^2_\B)\cap  C^1((0,\infty), L_2).
 \]
 Moreover, we have 
 \[
f_\e\geq\e, \qquad g_\e\geq \e\qquad\text{for all }\ (t,x)\in (0,\infty)\times(0,L).
\]
\end{thm}
%%%%%%%%%%%%%%%%%%%%%%%%%%%%%%%%%%%%%%%%

In order to prove this global result, we establish the following lemma which gives a criterion for global existence of classical solutions of \eqref{eq:RS}:

%%%%%%%%%%%%%%%%%%%%%%%%%%%%%%%%%%%%%%%%
\begin{lemma}\label{L:1} 
 Given $\e\in(0,1)$, the problem \eqref{eq:RS} possesses a unique maximal strong solution $X_\e=(f_\e,g_\e)$ on a maximal interval $[0,T_+(\e))$ satisfying 
\[
f_\e, g_\e\in C([0,T_+(\e)), H^1)\cap C((0,T_+(\e)), H^2_\B)\cap  C^1((0,T_+(\e)), L_2).
\]
Moreover, if for every $T<T_+(\e)$ there exists $C(\e,T)>0$  such that
\begin{equation}\label{eq:GC}
f_\e\geq\e/2+C(\e,T)^{-1},\quad  g_\e\geq\e/2+C(\e,T)^{-1},\quad\text{and}\quad \max_{t\in[0,T]}\|X_\e(t)\|_{H^1}\le C(\e,T),
\end{equation}
then the solution is globally defined, i.e.  $T_+(\e)=\infty.$
 \end{lemma}
%%%%%%%%%%%%%%%%%%%%%%%%%%%%%%%%%%%%%%%%

\begin{proof} Let $\e\in(0,1)$ be fixed. 
To lighten our notation we omit the subscript $\e$ in the remainder of this proof. 
Note first that problem \eqref{eq:RS}  has a  quasilinear structure, in the sense that \eqref{eq:RS} is equivalent to the system of equations:
 \begin{equation}\label{eq:RS1}
\left\{
\begin{array}{rllll}
\p_tX+\A(X)X&=&F(X)&\text{in}& (0,\infty)\times (0,L),\\[1ex]
\B X&=&0&\text{on}& (0,\infty)\times \{0,L\},\\[1ex]
X(0)&=&X_{0}&\text{on}& (0,L),
\end{array}
\right.
\end{equation}
where the new variable is $X:=(f,g)$ with $X_0=(f_0,g_0),$ and the operators $\B$ and $F$ are respectively given by
\[
\B X:=\p_x X, \qquad 
F(X):=
\p_x\left(
\begin{array}{ll}
R(f-\e)\p_x{G}\\[1ex]
R_\mu  (g-\e)\p_x{F}
\end{array}
\right).
\]
Letting 
\[
a(X):=
\left(
\begin{array}{ll}
{ (1+R)} f&0\\[1ex]
0&R_\mu  g
\end{array}
\right),
\]
the operator $\A$ is defined by the relation $\A(X)Y:=-\p_x(a(X)Y).$  We shall prove first that  \eqref{eq:RS1} has a weak solution defined on a maximal time interval, for which we have a weak criterion for global existence. We then improve in successive steps the regularity of the solution to show that it is actually a classical solution, so that this criterion guarantees also global existence of classical solutions.

Given $\alpha\in[0,1]$, the complex interpolation space $H^{2\alpha}_\B:=[L_2, H^2_\B]_\alpha$ is known to satisfy, cf. \cite{Am2},
\[
H^{2\alpha}_\B=
\left\{
\begin{array}{lll}
H^{2\alpha}&,& \alpha\leq 3/4,\\[1ex]
\{f\in H^{2\alpha}\, :\,\text{ $\partial_x f=0$ for $x=0,L$}\}&,&\alpha>3/4.
\end{array}
\right.
\]
Furthermore, for each $\beta\in(1/2, 2]$
we define the set 
\[
\V^\beta_\B:=\{f\in H^\beta_\B\,:\, \e/2<f\},
\]
which is open in $H^2_\B.$
Choose now $\gamma:=1/2-2\xi>0,$ where $\xi\in(0,1/18).$
We infer   from  \eqref{eq:A3}, that $X_0\in\V^{1-\xi}_\B\times\V^{1-\xi}_\B$. 
In order to obtain existence of a unique  weak solution of \eqref{eq:RS1} we verify the assumptions of \cite[Theorem 13.1]{Am2}. 
With the notation of \cite[Theorem 13.1]{Am2} we define 
\[
(\sigma,s,r,\tau):=(3/2-3\xi, 1+\xi,1-\xi,-\xi), \qquad 2\wh\alpha:=3/2-3\xi.
\]
Since $1-\xi-1/2>\gamma,$ we conclude that $\V^{1-\xi}_\B\subset C^\gamma([0,L]),$ meaning that the elements of $a(X)$ belong to  $C^\gamma([0,L])$ for $X\in \mathcal{V}^{1-\xi}_\B\times \mathcal{V}^{1-\xi}_\B$. Since $a(X)$ is  positive definite, we conclude in virtue of $\gamma>2\wh\alpha-1$ that 
\[
(\A,\B)\in C^{1-}(\mathcal{V}^{1-\xi}_\B\times \mathcal{V}^{1-\xi}_\B, \E^{\wh\alpha}(0,L)),
\]
the notation $\E^{\wh\alpha}(0,L)$ being defined in \cite[Sections~4 \&~8]{Am2}.
 Moreover, since $(1-\e^2\p_x^2)^{-1}\in\kL( L_2,H^2_\B)$ we also have
\[
F\in C^{1-}(\mathcal{V}^{1-\xi}\times \mathcal{V}^{1-\xi}, H^{-\xi}),
\]
whereby, in the notation of \cite{Am2}, $H^{-\xi}=H^{-\xi}_\B$ because $|\xi|<1/2,$   cf. \cite[eq. (7.5)]{Am2}.
Thus, we find all the assumptions of  \cite[Theorem 13.1]{Am2} fulfilled, and conclude that,
 for each $X_0\in \V^2_\B$, there exists a unique maximal weak  $H^{3/2-\xi}_\B-$solution $X=(f,g)$ of \eqref{eq:RS1}, that is 
\[
f,g \in C([0,T_+), \V^{1-\xi}_\B)\cap C((0,T_{+}), H^{3/2-\xi}_\B)\cap C^1((0,T_{ +}), H^{-1/2-3\xi}_\B),
\]
 and the first equation of \eqref{eq:RS1} is satisfied for all $t\in(0,T_+)$ when testing with functions belonging to $H^{1/2+3\xi}_\B.$
 Moreover, if $X\big|[0,T]$ is bounded in $H^1_\B\times H^1_\B$ and bounded away from  $\p\V^{1-\xi}_\B$ for all $T>0,$ { then} $T_+=\infty,$ which yields the desired criterion \eqref{eq:GC}. 
 
 We show now that this weak solution  has even more regularity, to conclude in the end that the existence time of the strong solution of \eqref{eq:RS1}
 coincides with that of the weak solution (of course they are identical on each interval where they are defined).
 Indeed, given $\delta>0,$
 it holds that $X\in C([\delta,T_+), H^{3/2-3\xi}_\B)\cap C^{1}([\delta,T_+), H^{-1/2-3\xi}_\B).$
 Hence, if $0<2\rho<2,$ we conclude from \cite[Theorem 7.2]{Am2}  and \cite[Proposition 1.1.5]{L} that $X$ is actually H\"older continuous 
 \[
X\in C^\rho([\delta,T_+), H^{3/2-3\xi-2\rho}_\B).
\]
 Choosing $\rho:=\xi$ and $\mu:=1-6\xi>0,$ we have that $H^{3/2-3\xi-2\rho}_\B\hookrightarrow C^\mu([0,L]),$ so that the elements of
 the matrix $a(X(t))$ belong to $C^{\mu}([0,L])$ for all $t\in[\delta,T_+).$
 Defining $2\wh \mu:=2-8\xi>0$, we observe that $\mu>2\wh \mu-1 $ and 
\[
(\A(X),\B)\in C^\rho([\delta, T_+),\mathcal{E}^{\wh \mu}((0,L))).
\]  
 Finally, with $2\nu:=3/2+\xi$,  our choice for $\xi$ implies
\[
(X(\delta), F(X))\in H^{2\nu-2}_\B\times C^\rho([\delta,T_+), H^{-\xi}_\B)\subset H^{2\nu-2}_\B\times C^\rho([\delta,T_+), H^{2\nu-2}_\B),
\]
and the assertions of \cite[Theorem 11.3]{Am2} are all fulfilled. 
Whence, the linear problem
 \begin{equation}\label{eq:Gse}
\left\{
\begin{array}{rllll}
\p_tY+\A(X)Y&=&F(X)&\text{in}& (\delta,{ T_+})\times (0,L),\\[1ex]
\B Y&=&0&\text{on}& (\delta,{ T_+})\times \{0,L\},\\[1ex]
Y(\delta)&=&X(\delta)&\text{on}& (0,L),
\end{array}
\right.
\end{equation}
possesses a unique strong $H^{2\nu}-$solution $Y$, that is
\[
Y \in C([\delta,T_+), H^{2\nu-2}_\B\times H^{2\nu-2}_\B)\cap C((\delta,T_+), H^{2\nu}_\B\times H^{2\nu}_\B)\cap C^1((\delta,T_+), H^{2\nu-2}_\B\times H^{2\nu-2}_\B).
\]
In view of \cite[Remark~11.1]{Am2} we conclude that both $X$ and $Y$ are weak $H^{3/2-3\xi}_\B-$solutions of \eqref{eq:Gse}, whence we infer from \cite[ Theorem~11.2]{Am2}    that $X=Y,$ and so
 \[
f, g\in C([\delta,T_+), H^{-1/2+\xi}_\B)\cap C((\delta,T_+), H^{3/2+\xi}_\B)\cap C^1((\delta,T_+), H^{-1/2+\xi}_\B).
\]
Interpolating as we did  previously and taking into account that $\delta$ was arbitrarily chosen, we have $f,g\in C^\theta([\delta, T_+), C^{1+\theta}([0,L]))$ if we set $4\theta\leq\xi.$ 
Hence, we find that
\[
(X(\delta), (A(X),F(X)))\in L_2\times C^\theta([\delta,T_+), \mathcal{H}(H^2_\B\times H^2_\B, L_2\times L_2)\times (L_2\times L_2)),
\] 
where $\mathcal{H}(H^2_\B\times H^2_\B, L_2\times L_2)$ denotes the set of linear operators in $L_2\times L_2$ with domain $H^2_\B\times H^2_\B$ which are negative infinitesimal generators of  analytic semigroups on $L_2\times  L_2$, and, in virtue of \cite[Theorem 10.1]{Am2}, 
 \[
X\in    C((\delta,T_+), H^{2}_\B\times H^{2}_\B)\cap C^1((\delta,T_+), L_2\times L_2)
\]
 for all $\delta\in(0,T_+).$
 Hence, the strong solution of \eqref{eq:RS1}, which is obtained by applying \cite[Theorem 12.1]{Am2} to that particular system,
 exists on $[0,T_+)$ and the proof is complete.
 \end{proof}
 
 The remainder of this section is devoted to prove that $T_+(\e)=\infty$ for the strong solution $(f_\e,g_\e)$ of \eqref{eq:RS} constructed in Lemma~\ref{L:1}. In view of Lemma \ref{L:1}, it suffices to prove that $(f_\e,g_\e)$ are a priori bounded in $H^1$ and away from zero. Concerning the latter, we note that we may apply the parabolic maximum principle to each equation of the regularised system \eqref{eq:A1} separately. 
Indeed, owing to the boundary conditions \eqref{A:2}, the constant function $(t,x)\mapsto \e$ solves the first equation of \eqref{eq:A1} and \eqref{A:2} while we have $f_{0\e}\ge \e$ by \eqref{eq:A3}. 
Consequently, $f_\e\ge \e$ in $[0,T_+(\e))\times [0,L]$ and, using a similar argument for $g_\e$, we conclude that
\begin{equation}\label{eq:GE1}
f_\e\geq\e, \qquad g_\e\geq \e\qquad \text{for all $(t,x)\in[0,T_+(\e))\times (0,L)$.} 
\end{equation}
 Next, owing to \eqref{A:2} and the nonnegativity of $f_\e$ and $g_\e$, it readily follows from \eqref{eq:A1} that the $L_1-$norm of $f_\e$ and $g_\e$ is conserved in time, that is,
\begin{equation}\label{eq:GE2}
\|f_\e(t)\|_1= \|f_{0\e}\|_1=\|f_0\|_1+\e L, \quad \|g_\e(t)\|_1=\|g_{0\e}\|_1=\|g_0\|_1+\e L 
\end{equation}
for all $t\in[0,T_+(\e))$. 
The next step is to improve the previous $L^1-$bound to an $H^1-$bound as required by Lemma~\ref{L:1}. 
To this end, we shall use the energy $\mathcal{E}_1$ for the regularised problem, see \eqref{eq:GE3} below. As a preliminary step, we collect some properties of the functions $(F_\e,G_\e)$ defined in \eqref{spirou} in the next lemma.

%%%%%%%%%%%%%%%%%%%%%%%%%%%%%%%%%%%%%%%%
\begin{lemma}\label{L:2.5} For all $t\in(0,T_+(\e))$
\begin{eqnarray}
& & \|F_\e(t)\|_2 \le \|f_\e(t)\|_2, \quad \|G_\e(t)\|_2 \le \|g_\e(t)\|_2 , \label{eq:es1s} \\
& & \|\p_x F_\e(t)\|_2 \leq \|\p_xf_\e(t)\|_2, \quad \|\p_x G_\e(t)\|_2 \leq \|\p_xg_\e(t)\|_2 , \label{eq:es1}\\
& & \e\ \|\p_x^2 F_\e(t)\|_2 \leq \|\p_xf_\e(t)\|_2, \quad \e\ \|\p_x^2 G_\e(t)\|_2 \leq \|\p_xg_\e(t)\|_2 . \label{eq:es1s1}
\end{eqnarray} 
\end{lemma}
%%%%%%%%%%%%%%%%%%%%%%%%%%%%%%%%%%%%%%%%

\begin{proof}
The proof of \eqref{eq:es1s} is similar to that of \eqref{eq:id1}. We next multiply the equation $F_\e - \e^2 \p_x^2 F_\e = f_\e$ by $- \p_x^2 F_\e$, integrate over $(0,L)$ and use the Cauchy-Schwarz inequality to estimate the right-hand side and obtain \eqref{eq:es1} and \eqref{eq:es1s1}. 
\end{proof}

%%%%%%%%%%%%%%%%%%%%%%%%%%%%%%%%%%%%%%%%
\begin{lemma}\label{L:3} Given $T\in(0,T_+(\e)),$ we have that
\begin{equation}\label{eq:GE3}
\E_1(f_\e(T), g_\e(T))+\int_0^T\int_0^L \left( \frac{1}{2} |\p_xf_\e|^2+ \frac{R}{1+2R} |\p_x g_\e|^2 \right) \, dx\, dt\leq \E_1(f_\e(0), g_\e(0)).
\end{equation}
\end{lemma}
%%%%%%%%%%%%%%%%%%%%%%%%%%%%%%%%%%%%%%%%

\begin{proof} Using \eqref{eq:es1} and H\"older's inequality, we get
\begin{align*}
\frac{d}{dt}\E_1(f_\e, g_\e)&=\int_0^L\p_t \left( f_\e\log(f_\e) \right) +\frac{R}{R_\mu}\p_t \left( g_\e\log(g_\e) \right)\, dx &\\[1ex] =&-\int_0^L\left((1+R)|\p_xf_\e|^2+R\frac{f_\e-\e}{f_\e}\p_xf_\e\p_xG_\e+R\frac{g_\e-\e}{g_\e}\p_xg_\e\p_xF_\e+R|\p_xg_\e|^2\right)\, dx\\[1ex]
\leq& { - (1+R)\|\p_xf_\e\|_2^2 + R \|\p_x f_\e\|_2\|\p_x G_\e\|_2 + R \|\p_x g_\e\|_2 \|\p_x F_\e\|_2 - R \|\p_x g_\e\|_2^2}\\[1ex]
\leq& -\frac{1}{2}\|\p_xf_\e\|_2^2 - \left( \frac{1+2R}{2} \|\p_xf_\e\|_2^2-2R\|\p_xf_\e\|_2\|\p_xg_\e\|_2+R\|\p_xg_\e\|_2^2\right)\\[1ex]
\leq&-\frac{1}{2}\|\p_xf_\e\|_2^2-\frac{R}{1+2R}\|\p_xg_\e\|_2^2.
\end{align*}
Integrating with respect to time, we obtain the desired assertion.
\end{proof} 
 
 Since $z\ln z-z+1\geq0$ for all $z\in[0,\infty),$ relation \eqref{eq:GE3} gives a uniform estimate in $(\e,t)\in (0,1)\times (0,T_+(\e))$ of  $(\p_x f_\e,\p_x g_\e)$ in $L_2((0,T), L_2\times L_2)$  in dependence only of the initial 
 condition $(f_0,g_0)$. 
Indeed, on the one hand, since $\ln z \le z-1$ for all $z\in [0,\infty)$, we have
\begin{equation}\label{eq:eq}
 z\ln z-z+1\leq z(z-1) - (z-1) = (z-1)^2\qquad \text{for all $z\geq0$},
\end{equation}
so that
\begin{equation}\label{eq:M}
\E_1(f_{0\e},g_{0\e})\leq \int_{0}^L \left( (f_{0\e}-1)^2+\frac{R}{R_\mu}(g_{0\e}-1)^2 \right)\, dx\leq {C_1}
\end{equation}
for all $\e\in(0,1)$ by \eqref{eq:id1}. 
On the other hand, owing to the Poincar\'e-Wirtinger inequality and \eqref{eq:GE2}, we have
$$
\|f_\e(t)\|_2 \le \left\| f_\e(t) - \frac{1}{L} \|f_\e(t)\|_1 \right\|_2 + \frac{\|f_\e(t)\|_1}{\sqrt{L}} \le C\ \left( \|\p_x f_\e(t)\|_2 + 1 \right).
$$
A similar bound being available for $g_\e$, we infer from \eqref{eq:GE3}, \eqref{eq:M}, and the nonnegativity of $\mathcal{E}_1$ that, for $T\in (0,T_+(\e))$, 
\begin{eqnarray}
\int_0^T \left( \|f_\e(t)\|_{H^1}^2 + \|g_\e(t)\|_{H^1}^2 \right)\, dt & \le & C\ \int_0^T \left( 1 + \|\p_x f_\e(t)\|_2^2 +  \|\p_x g_\e(t)\|_2^2 \right)\, dt \nonumber \\
& \le & C\ \left[ T+\mathcal{E}_1(f_{0\e},g_{0\e}) \right] \le C_2(T). \label{spip}
\end{eqnarray} 
We next use this estimate to prove  that  the solution $(f_\e,g_\e)$ of \eqref{eq:RS} is bounded  in $L_\infty((0,T),H^1\times H^1) $  for all $T<T_+(\e)$. 
 While the estimates were independent of $\e$ up to now, the next ones have a strong dependence upon $\e$ which explains the need of a regularisation of the original system.

%%%%%%%%%%%%%%%%%%%%%%%%%%%%%%%%%%%%%%%%
\begin{lemma}\label{L:4} Given $(\e,T)\in (0,1)\times (0,T_+(\e)),$ there exists a constant $C(\e,T)>0$ such that the solution $(f_\e,g_\e)$ of { \eqref{eq:RS}} fulfills
\begin{equation}\label{eq:GEx}
\|f_\e(t)\|_{H^1}+\|g_\e(t)\|_{H^1}\leq C(\e,T)\qquad\text{for all $t\in [0,T].$}
\end{equation}
\end{lemma}
%%%%%%%%%%%%%%%%%%%%%%%%%%%%%%%%%%%%%%%%

\begin{proof} We prove first the bound for $f_\e$. Given $z\in\R,$ let $q(z):=z^2/2.$ 
With this notation, the first equation of \eqref{eq:A1} { reads}
\[
\p_t f_\e-(1+R)\p_x^2 q(f_\e)=R\p_x((f_\e-\e)\p_xG_\e).
\]
 Multiplying  this relation by $\p_t q(f_\e)$ and integrating {over $(0,L)$},  we get    
 \[
\int_0^L\p_t f_\e \p_t q(f_\e)\, dx-(1+R)\int_0^L\p_x^2 q(f_\e) \p_t q(f_\e) \, dx=R\int_0^L\p_x((f_\e-\e)\p_xG_\e)\p_t q(f_\e)\, dx.
\]
{ Using an} integration by parts and Young's inequality, we come to the following inequality
 \[
\|\sqrt{f_\e}\p_tf_\e\|_2^2+\frac{1+R}{2}\frac{d}{dt}\|\p_xq(f_\e)\|_2^2\leq\frac{1}{2}\|\sqrt{f_\e}\p_tf_\e\|_2^2+\frac{R^2}{2}\int_0^Lf_\e[\p_x((f_\e-\e)\p_xG_\e)]^2\, dx.
\]
Whence, we have shown that
 \begin{equation}\label{eq:cvc}
\|\sqrt{f_\e}\p_tf_\e\|_2^2+(1+R)\frac{d}{dt}\|\p_xq(f_\e)\|_2^2\leq R^2\int_0^L \left[  f_\e^3 (\p_x^2G_\e)^2 + f_\e (\p_xf_\e)^2 (\p_xG_\e)^2 \right] \, dx,
\end{equation}
and the second term on the right-hand side of \eqref{eq:cvc} may be estimated, in view of \eqref{eq:GE1}, by 
\begin{align*}
\int_0^Lf_\e(\p_xf_\e)^2(\p_xG_\e)^2\, dx\leq&\frac{1}{\e}\int_0^Lf_\e^2 (\p_xf_\e)^2(\p_xG_\e)^2\, dx=\frac{1}{\e}\int_0^L(\p_xq(f_\e))^2(\p_xG_\e)^2\, dx\\[1ex]
\leq&\frac{1}{\e}\|\p_xG_\e\|_\infty^2\|\p_xq(f_\e)\|_2^2.
\end{align*}
Now, since $G_\e$ is the solution of $G_\e-\e^2\p_x^2G_\e=g_\e$  with homogeneous Neumann  boundary conditions at $x=0, L,$ and $g_\e\geq\e$, the elliptic maximum principle guarantees that  $G_\e\geq\e.$ 
Hence, $-\e^2\p_x^2G_\e\leq g_\e$, and therefore,  for $x\in (0,L)$,
\[
-\e^2\p_xG_\e(x)\leq \int_0^xg_\e\, dx  \le \|g_\e\|_1 \qquad\text{and}\qquad \e^2\p_xG_\e(x)\leq \int_x^Lg_\e\, dx   \le \|g_\e\|_1,
\]
so that $\e^2\|\p_xG_\e\|_\infty\leq \|g_\e\|_1.$ In view of \eqref{eq:GE2}, we arrive at
\begin{align}\label{eq:RHS1}
\int_0^Lf_\e(\p_xf_\e)^2(\p_xG_\e)^2\, dx\leq&\frac{1}{\e^5}\|g_\e\|_1^2\|\p_xq(f_\e)\|_2^2\leq C(\e)\|\p_xq(f_\e)\|_2^2.
\end{align}

Concerning the first term on the right-hand side of \eqref{eq:cvc},  it follows from \eqref{eq:GE1} and  \eqref{eq:es1s1} that
\begin{align*}
\int_0^L f_\e^3 (\p_x^2G_\e)^2\, dx\leq \frac{1}{\e} \int_0^Lf_\e^4 (\p_x^2G_\e)^2\, dx\le \frac{4}{\e}\|q(f_\e)\|_\infty^2 \|\p_x^2 G_\e\|_2^2 \leq \frac{4}{\e^3}\|q(f_\e)\|_\infty^2 \|\p_x g_\e\|_2^2.
\end{align*}
In order to estimate $\|q(f_\e)\|_\infty$, we choose  $x_\e\in(0,L)$ such that $Lf_\e(x_\e)=\|f_\e\|_1.$
By the fundamental theorem of calculus, we get
\[
q(f_\e)(x)=q(f_\e)( x_\e)+\int_{ x_\e}^x\p_x q(f_\e)\, dx\leq\frac{1}{2L^2}\|f_\e\|_1^2+\sqrt{L}\|\p_xq(f_\e)\|_2\quad\text{for all $x\in(0,L)$.}
\] 
 Summarising, we obtain in view of \eqref{eq:cvc} and \eqref{eq:RHS1}, that
\begin{equation}\label{eq:cvcv}
\frac{d}{dt}\|\p_xq(f_\e)\|_2^2\leq C(\e)(1+\|{\p_x g_\e}\|_2^2)(1+\|\p_xq(f_\e)\|_2^2),
\end{equation}
and from  \eqref{spip} and \eqref{eq:cvcv}
 \[
\|\p_xq(f_\e(t))\|_2^2\leq (1+\|\p_xq(f_{0\e})\|_2^2) \exp\left({C(\e)\int_0^t \left( 1+ \|\p_x g_\e\|_2^2 \right)\, ds}\right)\leq C(\e,T), \quad  t\in [0,T].
\] Since $f_\e\geq\e$, we then have
 \[
\|\p_xf_\e(t)\|_2^2\leq C(\e,T) \quad\text{for all $t\in[0,T]$.}
\]
Using  \eqref{eq:GE2} and the Poincar\'e-Wirtinger inequality, we finally obtain that
\[
\|f_\e(t)\|_{H^1}\leq C(\e,T) \quad\text{for all $t\in[0,T]$.}
\] 
Moreover, due to the symmetry of \eqref{eq:A1},  $g_\e$ satisfies the same estimate as $f_\e,$  and this completes our argument.
 \end{proof}

 \begin{proof}[Proof of Theorem~\ref{T:3}] The proof is a direct consequence of Lemma~\ref{L:1},  the lower bounds \eqref{eq:GE1}, and Lemma~\ref{L:4}.
 \end{proof}
 
 We end this section by showing that, though $\mathcal{E}_2$ is not dissipated along the trajectories of the regularised system \eqref{eq:RS}, a functional closely related to $\mathcal{E}_2$ is almost dissipated, with non-dissipative terms of order $\e$. 
 
\begin{lemma}\label{L:4.5} For $\e\in (0,1)$ and $T>0$, we have 
\begin{eqnarray}
\mathcal{E}_{2,\e}(f_\e(T),g_\e(T)) & + & \int_0^T \left[ f_\e \left| (1+R) \p_x f_\e + R \p_x G_\e \right|^2 + RR_\mu g_\e \left| \p_x(F_\e+g_\e) \right|^2 \right]\, dt \nonumber \\
& \le & \mathcal{E}_{2,\e}(f_{0\e},g_{0\e}) + \e C_2\ \int_0^T \varrho_\e(t)\, dt, \label{gaston}
\end{eqnarray}
with $\varrho_\e := \|\p_x f_\e\|_2^2 + \|\p_x g_\e\|_2^2$ and
\begin{equation}
2\mathcal{E}_{2,\e}(f_\e,g_\e) := (1+R) \|f_\e\|_2^2 + R \|g_\e\|_2^2 + R \int_0^L \left( F_\e g_\e + G_\e f_\e \right)\, dx. \label{fantasio} 
\end{equation}
\end{lemma}
 
\begin{proof}
 We multiply the first equation of \eqref{eq:RS} by $(1+R) f_\e + R G_\e$ and integrate over $(0,L)$ to obtain
\begin{equation}
\int_0^L \p_t f_\e\left( (1+R)f_\e+RG_\e \right)\, dx = -\int_0^L f_\e \left( (1+R)\p_x f_\e + R \p_xG_\e \right)^2\, dx + I_{1,\e}
\label{eq:zz1}
\end{equation}
with
$$
I_{1,\e}:=\e R\int_0^L \p_x G_\e \left( (1+R) \p_x f_\e + R \p_x G_\e\right)\, dx.
$$
Thanks to H\"older's inequality and \eqref{eq:es1}, we have
\begin{equation}
|I_{1,\e}| \le \e R(1+R) \left( \|\p_x G_\e\|_2 \|\p_x f_\e\|_2 + \|\p_x G_\e\|_2^2 \right) \le \e C \left( \|\p_x f_\e\|_2^2 + \|\p_x g_\e\|_2^2 \right)\,.
\label{eq:zz2}
\end{equation}
Similarly, multiplying the second equation of \eqref{eq:RS} by $R (F_\e+g_\e)$ and integrating over $(0,L)$ give
\begin{equation}
R\int_0^L \p_t g_\e \left( F_\e + g_\e \right)\, dx = - RR_\mu \int_0^L g_\e \left( \p_x F_\e + \p_x g_\e \right)^2\, dx + I_{2,\e}
\label{eq:zz3}
\end{equation}
with 
$$
I_{2,\e}:=\e RR_\mu\int_0^L\p_x F_\e \p_x\left( F_\e + g_\e \right)\, dx.
$$
Using again H\"older's inequality and \eqref{eq:es1}, we obtain
\begin{equation}
|I_{2,\e}| \le \e R R_\mu \left( \|\p_x F_\e\|_2^2 + \|\p_x g_\e\|_2 \|\p_x F_\e\|_2 \right) \le \e C \left( \|\p_x f_\e\|_2^2 + \|\p_x g_\e\|_2^2 \right)\,.
\label{eq:zz4}
\end{equation}
Observing that 
\begin{align*}
&  \int_0^L \p_t f_\e\left( (1+R)f_\e+RG_\e \right)\, dx + R\int_0^L \p_t g_\e \left( F_\e + g_\e \right)\, dx \\
& \phantom{spacespace}=  \frac{1+R}{2} \frac{d}{dt} \|f_\e\|_2^2 + R \int_0^L G_\e \left( \partial_t F_\e - \e^2 \p_x^2 \p_t F_\e \right)\, dx \\
&  \phantom{spacespace=}+ \frac{R}{2} \frac{d}{dt} \|g_\e\|_2^2 + R \int_0^L F_\e \left( \partial_t G_\e - \e^2 \p_x^2 \p_t G_\e \right)\, dx \\
& \phantom{spacespace}=  \frac{1}{2} \frac{d}{dt} \left( (1+R) \|f_\e\|_2^2 + R \|g_\e\|_2^2 + 2R \int_0^L F_\e G_\e\, dx \right) \\
&  \phantom{spacespace=}+ \e^2 R \int_0^L \left( \p_x G_\e \p_t \p_x F_\e + \p_x F_\e \p_t \p_x G_\e \right)\, dx \\
&\phantom{spacespace} =  \frac{1}{2} \frac{d}{dt} \left( (1+R) \|f_\e\|_2^2 + R \|g_\e\|_2^2 + 2R \int_0^L \left( F_\e G_\e + \e^2 \p_x F_\e \p_x G_\e \right)\, dx \right), 
\end{align*}
we sum \eqref{eq:zz1} and \eqref{eq:zz3}, use \eqref{eq:zz2} and \eqref{eq:zz4} to obtain
\begin{align*}
&  \frac{1}{2} \frac{d}{dt} \left( (1+R) \|f_\e\|_2^2 + R \|g_\e\|_2^2 + 2R \int_0^L \left( F_\e G_\e + \e^2 \p_x F_\e \p_x G_\e \right)\, dx \right) \\
& \phantom{spacespace}\le  - \int_0^L f_\e \left( (1+R)\p_x f_\e + R \p_xG_\e \right)^2\, dx - RR_\mu \int_0^L g_\e \left( \p_x F_\e + \p_x g_\e \right)^2\, dx + \e C_2\ \varrho_\e.
\end{align*}
Since 
\begin{eqnarray*}
2 \int_0^L \left( F_\e G_\e + \e^2 \p_x F_\e \p_x G_\e \right)\, dx & = & \int_0^L \left( 2 F_\e G_\e - \e^2 G_\e \p_x^2 F_\e -\e^2 F_\e \p_x^2 G_\e \right)\, dx \\ 
& = & \int_0^L \left( G_\e f_\e + F_\e g_\e \right)\, dx ,
\end{eqnarray*}
the claimed inequality follows from the above two identities after integration with respect to time.
\end{proof}

%%%%%%%%%%%%%%%%%%%%%%%%%%%%%%%%%%%%%%%%
%%%%%%%%%%%%%%%%%%%%%%%%%%%%%%%%%%%%%%%%
\section{ Weak solutions}\label{sec:2}
%%%%%%%%%%%%%%%%%%%%%%%%%%%%%%%%%%%%%%%%
%%%%%%%%%%%%%%%%%%%%%%%%%%%%%%%%%%%%%%%%

Given $T\in(0,\infty]$, we let $Q_T:=(0,T)\times(0,L).$ Furthermore, given $\e\in(0,1),$ we let $(f_\e,g_\e)$ be the global { strong} solution  of the regularised problem \eqref{eq:RS}  constructed in Theorem~\ref{T:3}. We shall prove that $(f_\e, g_\e)$ converges, in appropriate function spaces over $Q_T$, towards  { a pair of functions} $(f,g)$ which turns out to be a weak solution  of  \eqref{eq:S2} in the sense of Theorem~\ref{T:1}.

Recall that, by  \eqref{eq:GE1},  \eqref{eq:id1}, \eqref{eq:GE2}, and \eqref{eq:GE3}, $(f_\e,g_\e)$ satisfies the following estimates
\begin{equation}\label{eq:global}
\begin{aligned}
(a)&\quad f_\e\geq\e, \qquad g_\e\geq \e \qquad\text{on $Q_\infty$},\\[1ex]
(b)&\quad \|f_{0\e}\|_{2}+\|g_{0\e}\|_{2}\leq\|f_{0}\|_{2}+\|g_0\|_{2}+2{ \sqrt{L}},\\[1ex]
(c)&\quad \|f_\e(t)\|_1=\|f_0\|_1+\e L, \qquad \|g_\e(t)\|_1=\|g_0\|_1+\e L, \\[1ex]
(d)&\quad \E_1(f_\e(t), g_\e(t))+\int_0^t \left( \frac{1}{2} \|\p_xf_\e\|_2^2+\frac{R}{1+2R} \|\p_x g_\e\|_2^2 \right)\, ds\leq \E_1(f_\e(0), g_\e(0)),
\end{aligned}
 \end{equation}
{ for $t\ge 0$.}
Using \eqref{eq:global}, we show that:

%%%%%%%%%%%%%%%%%%%%%%%%%%%%%%%%%%%%%%%%
\begin{lemma}[Uniform estimates]\label{L:5} Let $h\in\{f,g,F, G\}.$
There exists a positive constant { $C_4(T)$} such that, for all $(\e,T)\in(0,1)\times(0,\infty),$ we have
\begin{align}
\label{eq:UE1}
(i)\quad&\int_0^T { \left( \| h_\e(t)\|_{H^1}^2 + \|h_\e(t)\|_3^3 \right)\, dt \leq C_4(T)},\\[1ex]
\label{eq:UE2}
(ii)\quad&\int_0^T\|\p_t { h_\e(t)}\|_{(W^1_6)'}^{6/5}\, dt\leq {C_4(T)}.
\end{align}
\end{lemma}
%%%%%%%%%%%%%%%%%%%%%%%%%%%%%%%%%%%%%%%%

\begin{proof}  The estimate for $h_\e$ in $L_2(0,T,H^1)$ is obtained from  the energy estimate \eqref{eq:GE3}, by taking also into account relations \eqref{eq:GE2}, \eqref{eq:es1s}, \eqref{eq:es1}, \eqref{eq:M}, the nonnegativity of $\mathcal{E}_1$, and the Poincar\'e-Wirtinger inequality as in the proof of \eqref{spip}.
 In order to prove the second estimate of \eqref{eq:UE1}, we note  that, since $H^1$ is continuously embedded in $L_\infty$ and $(f_\e)$ is bounded in $L_2(0,T,H^1)$, $(f_\e)$ is uniformly bounded with respect to $\e$  in $L_\infty(0,\infty; L_1)\cap L_2(0,T; L_\infty)$ for all $T>0.$ 
 The claimed $L_3-$bound then follows from the inequality $\|f_\e\|_3^3 \le \|f_\e\|_\infty^2 \|f_\e\|_1$. 
Next, an obvious consequence of the definition of $F_\e$ is that $\|F_\e\|_p\le\|f_\e\|_p$ for $p\in [1,\infty]$, from which we deduce the expected bound in $L_3$ for $(F_\e)$. A similar argument shows that $(g_\e)$ and $(G_\e)$  satisfy the second estimate in \eqref{eq:UE1}. 

In order to prove $(ii)$, consider first $h\in\{f,g\}.$ From \eqref{eq:UE1} and H\"older's inequality we obtain that $(f_\e \p_x f_\e)$, $((f_\e-\e)\p_x G_\e)$, $((g_\e-\e)\p_x F_\e)$, and $(g_\e \p_x g_\e)$ are uniformly bounded in $L_{6/5}(Q_T)$. Therefore, the  equations of \eqref{eq:A1}  may be written in the form  $\p_t h_\e=\p_x H^h_\e,$ for some function $H^h _\e$ which is uniformly bounded in $L_{6/5}(Q_T)$ and satisfies homogeneous  Dirichlet conditions $H^h _\e(0)=H^h _\e(L)=0.$ Given $\phi\in W^1_6,$  we have 
\begin{align*}
\left|\langle \p_th_\e|\phi\rangle_{L_2}\right|=&\left|\int_0^L\phi \p_xH^h_\e\, dx\right|=\left|\int_0^LH^h_\e\p_x\phi\, dx\right|\leq \|H^h_\e\|_{6/5} \ \|\p_x\phi\|_6.
\end{align*}
{ Consequently,}
\begin{align*}
\|\p_t h_\e(t)\|_{\left(W^1_6\right)'}\leq \|H^h_\e(t)\|_{6/5}\qquad\text{ { for} $t\in(0,T).$}
\end{align*}
and the families $(f_\e)$, $(g_\e)$ are both uniformly bounded in $L_{6/5}(0,T;\left(W^1_6\right)')$.

Finally, we have $\langle (1-\e^2\p_x^2 )p|q\rangle_{L_2}=\langle p|(1-\e^2\p_x^2 ) q\rangle_{L_2}$ for all $p,q\in H^2_\B$, 
and choosing $p:=(1-\e^2\p_x^2)^{-1}\p_t f_\e$, $q=(1-\e^2\p_x^2)^{-1}\phi$ with $\phi\in W^1_6$,
we have 
\begin{align*}
\langle\p_t F_\e|\phi\rangle_{L_2}=&\langle (1-\e^2\p_x^2)^{-1} \p_t f_\e|\phi\rangle_{L_2}=\langle  \p_t f_\e|(1-\e^2\p_x^2)^{-1}\phi\rangle_{L_2}\\[1ex]
=&\langle  \p_xH^f_\e|(1-\e^2\p_x^2)^{-1}\phi\rangle_{L_2}=-\langle  H^f_\e|(1-\e^2\p_x^2)^{-1}\p_x\phi\rangle_{L_2}. 
\end{align*}
Since $(1-\e^2\p_x^2)^{-1}$ is a contraction in $L_6$, we obtain that
\[
|\langle\p_t F_\e|\phi\rangle_{L_2}|\leq \|H^f_\e\|_{6/5}\|(1-\e^2\p_x^2)^{-1}\p_x\phi\|_6\leq  \|H^f_\e\|_{6/5}\|\phi\|_{W^1_6},
\]
and the assertion $(ii)$, when $h=F,$ follows at once. Invoking { a similar argument for $(G_\e)$}, we complete the proof.
\end{proof}

This lemma enables us to use a result from \cite{JS} and show that $(f_\e),$ $(g_\e),$ $(F_\e),$ and $(G_\e)$ are relatively compact in $L_2(0,T; C^{\alpha}([0,L])),$
provided that $\alpha\in(0,1/2).$
This will allow us to identify a limit point for each of these sequences, and find in this way a candidate for solving \eqref{eq:S2}.
Indeed, we have:

%%%%%%%%%%%%%%%%%%%%%%%%%%%%%%%%%%%%%%%%
\begin{lemma}\label{L:6} Given $h\in\{f,g,F,G\},$ and $\alpha\in (0,1/2),$ there exists a subsequence $(h_{\e_k})$ of   $(h_\e)$ which converges strongly in $L_2(0,T; C^{\alpha}([0,L]))$. 
\end{lemma}
%%%%%%%%%%%%%%%%%%%%%%%%%%%%%%%%%%%%%%%%

\begin{proof} Invoking the Rellich-Kondrachov theorem,  we have the following sequence of embeddings
\[
H^1\hookrightarrow  C^{\alpha}([0,L])\hookrightarrow \left(W^1_6\right)',\qquad \alpha<1/2,
\]
with compact embedding $H^1\hookrightarrow  C^{\alpha}([0,L]).$ 
Furthermore, in view of Lemma~\ref{L:5}~$(i)$, the family $(h_\e)$ is uniformly bounded in $L_2(0,T;H^1),$ while, by { Lemma~\ref{L:5}~$(ii)$}, $(\p_th_\e)$ is uniformly bounded in 
 $L_1(0,T; \left(W^1_6\right)')$.
 Whence, the assumptions of \cite[Corollary 4]{JS} are all fulfilled, and we conclude that $(h_\e)$ is relatively compact in  $L_2(0,T; C^{\alpha}([0,L])).$
\end{proof}
 
%%%%%%%%%%%%%%%%%%%%%%%%%%%%%%%%%%%%%%%%
%%%%%%%%%%%%%%%%%%%%%%%%%%%%%%%%%%%%%%%%
\subsection{Construction of weak solutions}
%%%%%%%%%%%%%%%%%%%%%%%%%%%%%%%%%%%%%%%%
%%%%%%%%%%%%%%%%%%%%%%%%%%%%%%%%%%%%%%%%

Using the uniform estimates deduced at the beginning of this section, we { now establish the} existence of a weak solution of \eqref{eq:S2}. { Owing to Lemma~\ref{L:6}, there are $f,g, F, G\in L_2(0,T; C^{\alpha}([0,L]))$ such that, for $\alpha\in (0,1/2)$,}
\begin{equation}\label{eq:func}
 \begin{aligned}
 f_{\e_k}\to f,\quad  g_{\e_k}\to g,\quad  F_{\e_k}\to F,\quad  G_{\e_k}\to G\qquad \text{in $L_2(0,T; C^{\alpha}([0,L])).$}
 \end{aligned}
 \end{equation} 
 Furthermore, by Lemma~\ref{L:5} $(i)$, the subsequences $({ \p_x} f_{\e_k}),$ $({ \p_x} g_{\e_k}),$ $({ \p_x} F_{\e_k}),$ and $ ({ \p_x}G_{\e_k})$ are uniformly bounded in the Hilbert space $L_2(Q_T).$ 
Hence, we may extract  further subsequences (denoted again by $(f_{\e_k}), (g_{\e_k}), (F_{\e_k}),$ and $ (G_{\e_k})$) which converge weakly:
 \begin{equation}\label{eq:der}
 \begin{aligned}
 \p_x f_{\e_k}\rightharpoonup { \p_x} f,\quad  \p_xg_{\e_k}\rightharpoonup { \p_x} g,\quad  \p_xF_{\e_k}\rightharpoonup { \p_x} F,\quad  \p_x G_{\e_k}\rightharpoonup { \p_x} G\qquad \text{in $L_2(Q_T).$}
 \end{aligned}
 \end{equation}
In fact,  we  have that
\begin{align}\label{eq:ident}
f=F\quad\text{and}\quad g=G\qquad\text{a.e. { in $Q_T$}.}
\end{align}
  Indeed, \eqref{eq:ident} follows by multiplying the relation $F_{\e_k}-\e_k^2\p_x^2F_{\e_k}=f_{\e_k}$ by a test function in $H^1$, integrating by parts,  and letting then 
{ $k\to \infty$ with the help of \eqref{eq:func} and \eqref{eq:der}.}
In view of \eqref{eq:func}-\eqref{eq:ident},  we then have
 \begin{equation}\label{eq:derr}
 \begin{aligned}
&f_{\e_k} \p_x f_{\e_k}\rightharpoonup  f\p_x  f,\quad  f_{\e_k} \p_xG_{\e_k}\rightharpoonup  f\p_xg\qquad \text{in $L_1(Q_T),$}\\[1ex]
 &g_{\e_k}\p_x F_{\e_k}\rightharpoonup  g\p_xf ,\quad g_{\e_k}\p_xg_{\e_k}\rightharpoonup  g\p_x g\qquad \text{in $L_1(Q_T).$}
 \end{aligned}
 \end{equation}
 
Using the fact that  $(f_\e,g_\e)$ are strong solutions of \eqref{eq:RS}, we obtain by integration with respect to space and time that 
 \begin{equation}\label{eq:A1-}
\begin{aligned}
\int_0^L f_{\e_k}(T)\psi\, dx-\int_0^Lf_{0\e_k}\psi\, dx=& { -} \int_{Q_T}  \left[ (1+R) f_{\e_k} \p_x{f_{\e_k}} + R (f_{\e_k}-\e_k) \p_x{G_{\e_k}} \right]\ \p_x\psi\, dx\, dt,\\[1ex]
\int_0^L g_{\e_k}(T)\psi\, dx-\int_0^Lg_{0\e_k}\psi\, dx=& { -} R_\mu\int_{Q_T}  \left[  (g_{\e_k}-\e_k) \p_x{F_{\e_k}} + g_{\e_k} \p_x{g_{\e_k}} \right]\ \p_x\psi\, dx\, dt,
\end{aligned}
\end{equation}
 for all $T>0,$ $\e\in(0,1),$ and $\psi\in W^1_\infty.$ Since
\begin{equation}\label{eq:maam}
f_{0\e_k}\to f_0\quad\text{and}\quad g_{0\e_k}\to g_0\qquad\text{ in $L_2$}
\end{equation}
by classical arguments, we may pass to the limit as $k\to\infty$ in \eqref{eq:A1-} and use \eqref{eq:func}, \eqref{eq:der}, \eqref{eq:derr}, and \eqref{eq:maam} to conclude that $(f,g)$ is a weak solution of \eqref{eq:RS} in the sense of Theorem~\ref{T:1}. The fact that $(f,g)$ can be defined  globally follows by using a standard Cantor's diagonal argument (using a sequence $T_n\nearrow \infty$).
 
%%%%%%%%%%%%%%%%%%%%%%%%%%%%%%%%%%%%%%%%
%%%%%%%%%%%%%%%%%%%%%%%%%%%%%%%%%%%%%%%%
\subsection{Energy estimates for weak solutions}
%%%%%%%%%%%%%%%%%%%%%%%%%%%%%%%%%%%%%%%%
%%%%%%%%%%%%%%%%%%%%%%%%%%%%%%%%%%%%%%%%

 Letting $\e\to 0$ in the relation \eqref{eq:global} $(c)$, we find in view of \eqref{eq:func},
 that 
 \[
{ \|f(t)\|_1=\|f_0\|_1 ,\quad \|g(t)\|_1=\|g_0\|_1, \qquad t\in(0,\infty).}
\]

We show now that the weak solution found above satisfies the energy estimate
\begin{equation}\label{E:1}
\E_1(f(T),g(T))+\int_{Q_T} \left( \frac{1}{2}|\p_xf|^2+\frac{R}{1+2R}|\p_xg|^2 \right)\, dx\, dt\leq  \E_1(f_0,g_0)
\end{equation}
for  $T\in(0,\infty).$ 
Recall that, by Lemma~\ref{L:3}, we have  
\begin{align}\label{eq:he1}
&\E_1(f_{\e_k}(T), g_{\e_k}(T))+\int_{Q_T} \left( \frac{1}{2}|\p_xf_{\e_k}|^2+\frac{R}{1+2R}|\p_x g_{\e_k}|^2 \right)\, dx\, dt\leq\E_1(f_{0\e_k}, g_{0\e_k})
\end{align} 
 for all $k\in\N.$ 
On the one hand, note that  \eqref{eq:func} and Fatou's lemma ensure that
 \begin{equation}\label{eq:E20}
\E_1(f(T), g(T))\leq \liminf_{k\to\infty}\E_1(f_{\e_k}(T), g_{\e_k}(T))\qquad\text{for $T\in(0,\infty),$}
\end{equation}
 while  \eqref{eq:der}   implies 
\begin{equation}\label{eq:E21}
 \begin{aligned}
 \int_{Q_T}|\p_xf|^2\, dx\, dt&\leq\liminf_{k\to\infty}\int_{Q_T}|\p_xf_{\e_k}|^2\, dx\, dt,\\[1ex]
 \int_{Q_T}|\p_xg|^2\, dx\, dt&\leq\liminf_{k\to\infty}\int_{Q_T}|\p_xg_{\e_k}|^2\, dx\, dt.
 \end{aligned}
 \end{equation}
We still have to pass to the limit in the right-hand side of \eqref{eq:he1}. By \eqref{eq:maam}, we may assume that $(f_{0\e_k})$ and $(g_{0\e_k})$ converge almost everywhere towards $f_0$ and $g_0,$ respectively. Furthermore, since $0\leq x\ln x-x+1\leq 1+2x^{3/2}$ for  $x\geq0,$ we have
\[
\int_{E}|f_{0\e_k} \ln f_{0\e_k} - f_{0\e_k} + 1|\, dx\leq |E|+2\int_E f_{0\e_k}^{3/2}\, dx\leq |E|+2|E|^{1/4}\|f_{0\e_k}\|_2^{3/2}\leq C|E|^{1/4}
\] 
 for all $k\in\N$ and all measurable subsets $E$ of $(0,L),$ meaning that the family $( f_{0\e_k} \ln f_{0\e_k} - f_{0\e_k} + 1)$ is uniformly integrable. Clearly, the same is true also for $( g_{0\e_k} \ln g_{0\e_k} - g_{0\e_k} + 1).$ We infer then from  Vitali's convergence theorem, cf. \cite[Theorem 2.24]{FL}, that the limit of the right-hand side of \eqref{eq:he1} exists and
\[
\lim_{k\to\infty }\E_1(f_{0\e_k}, g_{0\e_k})=\E_1(f_{0}, g_{0}). 
\]
Whence, passing to the limit in \eqref{eq:he1}, we obtain in view of \eqref{eq:E20} and \eqref{eq:E21} the desired estimate \eqref{E:1}.

Finally, we show that weak solutions of \eqref{eq:S2} satisfy 
 \begin{equation}\label{E:2}
 \E_2(f(T),g(T))+\int_{Q_T} \left[ f\left((1+R)\p_xf+R\p_xg\right)^2+RR_\mu g(\p_xf+\p_xg)^2 \right] \, dx\, dt\leq \E_2(f_{0},g_{0})
 \end{equation}
for $T\in(0,\infty).$ In virtue of \eqref{eq:func}, \eqref{eq:der}, { and \eqref{eq:ident}} we have 
\[
\sqrt{f_{\e_k}}\to \sqrt{f}, \quad \p_x f_{\e_k}\rightharpoonup \p_x f, \quad\text{and}\quad { \p_x G_{\e_k}\rightharpoonup \p_x g}\qquad\text{in $L_2(Q_T)$}
\]  
 which implies that $\sqrt{f_{\e_k}}\p_xf_{\e_k}\rightharpoonup\sqrt{f}\p_xf$ { and $\sqrt{f_{\e_k}}\p_x G_{\e_k} \rightharpoonup \sqrt{f} \p_x g$} in $L_1(Q_T).$ 
{ Consequently,}
$$
\sqrt{f_{\e_k}}\left((1+R)\p_xf_{\e_k}+R\p_xG_{\e_k}\right)\rightharpoonup \sqrt{f}\left((1+R)\p_xf+R\p_xg\right)\qquad \text{in $L_1(Q_T).$}
$$
and, by a similar argument,
$$
\sqrt{g_{\e_k}}(\p_xF_{\e_k}+\p_xg_{\e_k})\rightharpoonup \sqrt{g}(\p_xf+\p_xg)\qquad \text{in $L_1(Q_T).$}
$$
{ Now, owing to \eqref{eq:UE1}, the sequence $(\varrho_{\e_k})$ defined in Lemma~\ref{L:4.5} is bounded in  $L_2((0,T))$ and we then infer from Lemma~\ref{L:4.5} that both $(\sqrt{f_{\e_k}}\left((1+R)\p_xf_{\e_k}+R\p_xG_{\e_k}\right))$ and $(\sqrt{g_{\e_k}}(\p_xF_{\e_k}+\p_xg_{\e_k}))$ are bounded in $L_2(Q_T)$. The previous weak convergences in $L_1(Q_T)$ may then be improved to weak convergence in $L_2(Q_T)$ (upon extracting a further subsequence if necessary) and we can then pass to the limit in \eqref{gaston} to conclude that \eqref{E:2} holds true, using weak lower semicontinuity arguments in the left-hand side and the property $\e_k \varrho_{\e_k}\to 0$ in  $L_2((0,T))$ in the right-hand side.}

%%%%%%%%%%%%%%%%%%%%%%%%%%%%%%%%%%%%%%%%
%%%%%%%%%%%%%%%%%%%%%%%%%%%%%%%%%%%%%%%%
\subsection{Exponential convergence towards equilibria} 
%%%%%%%%%%%%%%%%%%%%%%%%%%%%%%%%%%%%%%%%
%%%%%%%%%%%%%%%%%%%%%%%%%%%%%%%%%%%%%%%%

In this last part of the paper we prove our second main result, Theorem~\ref{T:2}.
 The proof is based on the  interplay between estimates for  the two energy functionals $\E_1$ and $\E_2,$
 with the specification that we use $\E_1$ to estimate the time derivative of the stronger energy functional $\E_2$,
 and obtain exponential decay of weak solutions in the $L_2-$norm.
 Recall from \eqref{eq:global}~ $(c)$  that 
\[
A_k:=\frac{\|f_0\|_1}{L}+\e_k=\frac{\|f_{\e_k}(t)\|_1}{L}\quad\text{ and}\quad B_k:=\frac{\|g_0\|_1}{L}+\e_k=\frac{\|g_{\e_k}(t)\|_1}{L}
\]
for all $k\in\N$ { and} $t\in[0,\infty).$ 
{ Introducing}
 \begin{align*}
\mathcal{F}_k:=&\int_0^L \left[ \left(f_{\e_k}\ln\left( \frac{f_{\e_k}}{A_k} \right) -f_{\e_k}+A_k\right)+\frac{R}{R_\mu}\left(g_{\e_k}\ln\left( \frac{g_{\e_k}}{B_k} \right)-g_{\e_k}+B_k\right) \right]\, dx\\[1ex]
&+\frac{1}{2}\int_0^L \left\{ (f_{\e_k}-A_k)^2+R\left[(f_{\e_k}-A_k)^2+(g_{\e_k}-B_k)^2+(g_{\e_k}-B_k)(F_{\e_k}-A_k)\right. \right.\\[1ex]
&\hspace{3cm}\left. \left.+(f_{\e_k}-A_k)(G_{\e_k}-B_k)\right] \right\}\, dx,
\end{align*}
{ we infer from \eqref{eq:global}~(c) and the proofs of Lemma~\ref{L:3} and~\ref{L:4.5} that}
 \begin{align*}
 \frac{d\mathcal{F}_k}{dt}=&\int_0^L\left[ \p_tf_{\e_k}\ln f_{\e_k}+\frac{R}{R_\mu}\p_tg_{\e_k}\ln g_{\e_k} \right]\, dx\\[1ex]
&+ \frac{1}{2} \frac{d}{dt}\int_0^L \left\{ (1+R)f_{\e_k}^2+R\left[g_{\e_k}^2+f_{\e_k}G_{\e_k}+g_{\e_k}F_{\e_k}\right] \right\}\, dx\\[1ex]
 \leq&-\frac{1}{2}\|\p_xf_{\e_k}\|_2^2-\frac{R}{1+2R}\|\p_x g_{\e_k}\|_2^2 +\e_k C_3 \left( \|\p_xf_{\e_k}\|_2^2+\|\p_xg_{\e_k}\|_2^2 \right)\\[1ex]
 \leq &-\frac{1}{3} \|\p_xf_{\e_k}\|_2^2-\frac{R}{2+2R}\|\p_x g_{\e_k}\|_2^2
 \end{align*}
 provided $k$ is large enough.
 Using { the Poincar\'e-Wirtinger} inequality, we find a positive constant $C_5$ such that
  \begin{equation}\label{eq:BB}
 \frac{d\mathcal{F}_k}{dt} \leq -C_5\ \left( \|f_{\e_k}-A_k\|_2^2+\| g_{\e_k}-B_k\|_2^2 \right)
 \end{equation}
 for large $k$ and all $t\in(0,\infty).$
We show now that the right-hand side of \eqref{eq:BB} can be bounded by $-\omega\mathcal{F}_k$ { for some small positive number $\omega$}. 
Indeed, { arguing as in Lemma~\ref{L:2.5}}, we find that
\begin{eqnarray}
& &\left| \int_0^L \left[ (g_{\e_k}-B_k)(F_{\e_k}-A_k)+(f_{\e_k}-A_k)(G_{\e_k}-B_k) \right]\, dx \right|\nonumber\\[1ex]
& & \phantom{space}{ \leq \|g_{\e_k}-B_k\|_2 \|F_{\e_k}-A_k\|_2 + \|f_{\e_k}-A_k\|_2 \|G_{\e_k}-B_k\|_2 }\nonumber\\[1ex]
& &\phantom{space}{ \leq 2\|g_{\e_k}-B_k\|_2 \|f_{\e_k}-A_k\|_2} \leq \|f_{\e_k}-A_k\|_2^2+\| g_{\e_k}-B_k\|_2^2.\label{eq:pp1} 
 \end{eqnarray}
{ Recalling \eqref{eq:eq}, we end up with}
 \begin{equation}\label{eq:pp2} 
\begin{aligned}
&\|f_{\e_k}-A_k\|_2^2+ \|g_{\e_k}-B_k\|_2^2 =\int_0^L \left( A_k^2\left|\frac{f_{\e_k}}{A_k}-1\right|^2+B_k\left|\frac{g_{\e_k}}{B_k}-1\right|^2 \right)\, dx\\[1ex]
&\phantom{space}\geq \int_0^L \left[ A_k^2\left(\frac{f_{\e_k}}{A_k}\ln\left( \frac{f_{\e_k}}{A_k} \right)-\frac{f_{\e_k}}{A_k}+1\right)+B_k^2\left(\frac{g_{\e_k}}{B_k}\ln\left( \frac{g_{\e_k}}{B_k} \right) -\frac{g_{\e_k}}{B_k}+1\right) \right]\, dx\\[1ex]
&\phantom{space}\geq \! \min_k\left\{A_k,\frac{R_\mu B_k}{R}\right\}\int_0^L \left[ \left(f_{\e_k}\ln\left( \frac{f_{\e_k}}{A_k} \right) -f_{\e_k}+A_k\right)+\frac{R}{R_\mu} \left(g_{\e_k}\ln\left( \frac{g_{\e_k}}{B_k} \right) - g_{\e_k}+B_k\right) \right]\, dx. 
 \end{aligned}
 \end{equation}
 Combining   \eqref{eq:BB}, \eqref{eq:pp1}, and \eqref{eq:pp2}, we conclude that if $\|f_0\|_1>0$ and $\|g_0\|_1>0,$  then
 \begin{equation*}
 \frac{d\mathcal{F}_k}{dt}(t) \leq -\omega\mathcal{F}_k(t)
 \end{equation*}
for some positive constant $\omega$ and $k$ sufficiently large.
Whence, 
\[
\|f_{\e_k}(t)-A_k\|_{L_2}+\|g_{\e_k}(t)-B_k\|_{L_2}\leq Ce^{-\omega t},
\] 
which yields, for $k\to\infty$, the desired estimate { by \eqref{eq:func} and \eqref{eq:ident}}, as stated in Theorem~\ref{T:2}.

If $f_0=0$ [resp. $g_0=0$], then $f=0$ [resp. $g=0$], while  $g$ [resp. $f$] is a weak solution of the one-dimensional  porous medium equation  and converges therefore even in the $L_\infty-$norm to
flat equilibria (if $f_0=0$, then uniqueness of solutions to \eqref{eq:RS1} implies that   $f_\e=\e$ and $\p_tg_\e=R_\mu(g_\e\p_xg_\e)$), cf. \cite[Theorem 20.16]{Vaz}.
Convergence in this stronger norm is due to the fact that comparison methods may be used for the one-dimensional porous media equation, while for our system they fail because of the structure of the system.\\[3ex]

\end{document}